\documentclass[11pt,a4paper, tikz, british]{amsart}

\usepackage{etoolbox}
\makeatletter
\patchcmd{\@maketitle}{\global\topskip42\p@\relax}
{\global\topskip42\p@\relax \vspace*{-64pt}}
{}{}
\makeatother

\usepackage{babel}
\usepackage[useregional]{datetime2}
\DTMlangsetup[en-GB]{en-GB}

\usepackage{amssymb}
\usepackage[hidelinks]{hyperref}

\usepackage{enumitem}

\newtheorem{thm}{Theorem}[section]
\newtheorem{prop}[thm]{Proposition}
\newtheorem{lem}[thm]{Lemma}

\theoremstyle{definition}
\newtheorem{definition}[thm]{Definition}

\theoremstyle{remark}
\newtheorem{remark}[thm]{Remark}

\numberwithin{equation}{section}

\newcommand{\alg}{\overline{\mathbb{Q}}} 
\newcommand{\R}{\mathbb{R}}  
\newcommand{\Q}{\mathbb{Q}} 
\newcommand{\Z}{\mathbb{Z}} 
\newcommand{\C}{\mathbb{C}} 
\newcommand{\Aa}{\mathbb{A}^1(\mathbb{C})} 
\newcommand{\AB}{\mathbb{A}^2(\mathbb{C})} 
\newcommand{\AN}{\mathbb{A}^n(\mathbb{C})} 
\newcommand{\h}{\mathbb{H}} 
\newcommand{\SL}{\mathrm{SL}_2(\Z)} 
\newcommand{\pr}{\mathbb{P}^1} 
\newcommand{\falt}{h_\mathrm{F}} 
\newcommand{\End}{\mathrm{End}} 
\newcommand{\lcm}{\mathrm{lcm}} 

\DeclareMathOperator{\re}{Re} 
\DeclareMathOperator{\im}{Im} 
\DeclareMathOperator{\cl}{cl} 

\begin{document}
	
	\title[Andr\'e's theorem and weakly bounded height]{Andr\'e's theorem and weakly bounded height}
	\author{Guy Fowler}
	\address{Department of Mathematics, University of Manchester, Manchester, UK
		\newline
		\indent
		Heilbronn Institute for Mathematical Research, Bristol, UK}
	\email{\href{mailto:guy.fowler@manchester.ac.uk}{guy.fowler@manchester.ac.uk}}
	\urladdr{\url{https://www.guyfowler.uk/}}
	\date{\today}
	
	\begin{abstract}
		Let $V \subset \AB$ be an algebraic curve such that $\deg X \neq \deg Y$, where $X, Y$ denote the coordinate functions on $\AB$ restricted to $V$. 
		We prove there exists an effectively computable constant $c$, that depends linearly on the height of $V$, such that $\max \{h(x), h(y)\} \leq c$ for every $(x, y) \in V$ with $x$ and $y$ both CM $j$-invariants.
		This establishes, for such curves, an effective version of the Andr\'e--Oort conjecture that has a better dependence on the height of $V$ than previous effective results.

	\end{abstract}
	
	\maketitle
	
	\section{Introduction}
	
A point $(x, y) \in \AB$ is a \textit{special point} if $x$ and $y$ are both \textit{singular moduli}, i.e.~$j$-invariants of elliptic curves with complex multiplication.
Denote by $\Delta_x$ (respectively $\Delta_y$) the discriminant of the endomorphism order of an elliptic curve with $j$-invariant $x$ (respectively $y$).

An irreducible algebraic curve $V \subset \AB$ is a \textit{special curve} if $V$ is one of:
\begin{enumerate}
	\item a horizontal line $\Aa \times \{y\}$ for some singular modulus $y$,
	\item a vertical line $\{ x\} \times \Aa$ for some singular modulus $x$,
	\item a modular curve $\mathbb{V}(\Phi_N(X, Y))$ for some $N \in \Z_{> 0}$, where $\Phi_N \in \Z[X, Y]$ is the modular polynomial of level $N$, see \cite[p.~55]{Lang87}.
\end{enumerate}
The curve $V$ is \textit{non-special} if it is not a special curve.

	Andr\'e \cite{Andre98} proved that every non-special algebraic curve $V \subset \AB$ contains only finitely many special points.
	This is the simplest non-trivial case of the Andr\'e--Oort conjecture for $\AN$; Pila \cite{Pila11} subsequently proved the full Andr\'e--Oort conjecture for $\AN$.
These proofs were ineffective: in particular, they do not give an effectively computable bound on the height (or even the number) of special points on 	a non-special curve $V$.
	
	The main unconditional result of this paper is an effective version of Andr\'e's theorem for a certain class of non-special curves. 
	In particular, we will prove that the height of special points on such curves $V$ may be bounded linearly in terms of the height of the curve.
	See \S\ref{sec:hts} for the definitions of the different notions of heights that appear in this paper.

	\begin{thm}\label{thm:main}
			Let $V \subset \AB$ be a geometrically irreducible, non-special algebraic curve defined over a number field $L$. Write $X, Y$ for the coordinate functions on $\AB$ restricted to $V$.
			Let $d_1 = \deg X$ and $d_2 = \deg Y$.
			Suppose that $d_1 \neq d_2$. 
			Then there exist an effectively computable constant $c_1([L : \Q], \max\{d_1, d_2\} )$ and, for every $\delta > 0$, an effectively computable constant $c_2( [L : \Q], \max\{d_1, d_2\}, \delta)$ 	
with the following properties: for every special point $(x, y) \in V$, 
\begin{align}\label{eq:htbd}
				\max \{ h(x), h(y)\} \leq c_1\left([L : \Q], \max\{d_1, d_2\}\right) + 26 \max \{d_1, d_2\}^3 h(V)
			\end{align}
			and
			\begin{align}\label{eq:discbd}
				\max \{ \lvert \Delta_x \rvert, \lvert \Delta_y \rvert\} \leq c_2([L : \Q], \max\{d_1, d_2\}, \delta) \max \{1, h(V)\}^{2 + \delta}.
			\end{align}
			 
	\end{thm}

The condition in Theorem~\ref{thm:main} that $d_1 \neq d_2$ implies, but is strictly stronger than, the condition that $V$ is not a modular curve $\mathbb{V}(\Phi_N(X, Y))$.
The condition that $V$ is a non-special curve is necessary, because every special curve contains infinitely many special points.
For non-special curves in general, K\"uhne \cite{Kuhne12} and, independently, Bilu, Masser, and Zannier \cite{BiluMasserZannier13} proved effective versions of Andr\'e's theorem.

 Compared to these previous effective versions of Andr\'e's theorem,  the novel aspect of Theorem~\ref{thm:main} is that, when applicable, it gives a significantly better dependence on the height of the curve $V$.
K\"uhne's result \cite[Theorem~2]{Kuhne12} gives a bound of the form of \eqref{eq:discbd}, but with exponent $8 + \delta$ in place of $2 + \delta$. 
W\"ustholz  \cite[Theorem~1.1]{Wustholz14} subsequently improved this exponent to $8$. 
The result of Bilu, Masser, and Zannier \cite{BiluMasserZannier13} does not give an explicit dependence on the height of the curve.
K\"uhne \cite[p.~654]{Kuhne12} notes that a bound as strong as \eqref{eq:discbd} for general non-special curves would follow via his method from a conjectural elliptic analogue \cite[Conjecture~1.7]{DavidHirata09} of the Lang--Waldschmidt conjecture.
Such conjectures seem out of reach at present however.

More recently, Binyamini \cite[Corollary~3]{Binyamini19} and Papas \cite[Corollary~1.7]{Papas26} have also given effective proofs of Andr\'e's theorem for certain classes of curves, but again without obtaining an explicit dependence on the height of the curve.
Their proofs both follow the Pila--Zannier strategy, which was originally applied to Andr\'e--Oort for $\AN$ by Pila \cite{Pila09a, Pila11}. 
Instead of the ineffective Pila--Wilkie theorem \cite{PilaWilkie06, Pila09}, Binyamini and Papas apply effective point counting results also due to Binyamini \cite{Binyamini19a, Binyamini24}.

\subsection{A conditional result and uniformity in Andr{\'e}'s theorem }

Our proof of Theorem~\ref{thm:main} will use a result of Habegger \cite[Theorem~1.1]{Habegger10}, that gives a height bound for points lying in the intersection of a modular curve $\mathbb{V}(\Phi_N(X, Y))$ and a curve $V$ satisfying the hypotheses of Theorem~\ref{thm:main}.
This is the only part of the proof of Theorem~\ref{thm:main} that requires the hypothesis that $d_1 \neq d_2$, rather than just the condition that $V$ is non-special.

Habegger conjectures \cite[p.~44]{Habegger10} that an analogous height bound to \cite[Theorem~1.1]{Habegger10} should hold solely under the hypothesis that $V$ is a non-special curve.
Assuming such a conjecture, we can extend the proof of Theorem~\ref{thm:main} to all non-special curves. 
In fact, a qualitatively weaker height bound than the one in \cite[Theorem~1.1]{Habegger10} would suffice.

\begin{definition}\label{def:WBH}
	Let $\eta \in (0, 1/2)$ and $\kappa \in \R_{\geq 0}$.
	We say that the property $\mathrm{WBH}(\eta, \kappa)$ holds if the following property is satisfied.
	Let $V \subset \AB$ be a geometrically irreducible, non-special algebraic curve defined over a number field $L$.
	Write $X, Y$ for the coordinate functions on $\AB$ restricted to $V$.
	Let $d_1 = \deg X$ and $d_2 = \deg Y$.
	There exists a constant $c_* = c_*(\eta, \kappa, [L : \Q], \max\{d_1, d_2\})$ such that, for every special point $(x, y) \in V$ and every $N \in \Z_{> 0}$, if $\Phi_N(x, y) = 0$, then
	\[ \max \{h(x), h(y)\} \leq c_* \max \{1, h(V)\}^\kappa N^{1/2 - \eta}.\]
\end{definition}

Similar height bounds have been formulated in \cite[Appendix~B]{Habegger17a} and \cite[Conjecture~21.24]{Pila22}.
The key aspect to note in the formulation of $\mathrm{WBH}(\eta, \kappa)$ is that it records the dependence on $h(V)$ explicitly.

The second main result of this paper is the following conditional result.

\begin{thm}\label{thm:WBH}
	Let $\eta \in (0, 1/2)$ and $\kappa \in \R_{\geq 0}$.
	Suppose that $\mathrm{WBH}(\eta, \kappa)$ holds with a given $c_*$.
	Let $V \subset \AB$ be a geometrically irreducible, non-special algebraic curve defined over a number field $L$.
	Write $X, Y$ for the coordinate functions on $\AB$ restricted to $V$.
	Let $d_1 = \deg X$ and $d_2 = \deg Y$.
	\begin{enumerate}[left=0pt]
		\item For every $\delta > 0$, there exists an effectively computable constant $c_1([L : \Q], \max \{d_1, d_2\}, c_*, \delta)$ such that
		\[ \quad \quad \quad \max \{h(x), h(y)\} \leq c_1([L : \Q], \max \{d_1, d_2\}, c_*, \delta) \max \{1, h(V)\}^{\kappa/2 \eta + \delta}\]
		for every special point $(x, y) \in V$.
		\item If $\kappa < 2 \eta$, then there exists an effectively computable constant $c_2([L : \Q], \max \{d_1, d_2\}, c_*)$ such that
		\[ \max \{\lvert \Delta_x \rvert, \lvert \Delta_y \rvert\} \leq c_2([L : \Q], \max \{d_1, d_2\}, c_*)\]
		for every special point $(x, y) \in V$.
	\end{enumerate}
	
\end{thm}

The motivation for part (2) of this theorem is the following ineffective result of Pila \cite{Pila09a}.
This is a uniform version of Andr\'e's theorem, in which there is no dependence on the height of the curve $V$.
A different proof was given, again ineffectively, by K\"uhne \cite[Theorem~2]{Kuhne13}.
Edixhoven \cite[Remark~7.1]{Edixhoven98} had earlier obtained the same result, conditional on GRH.

\begin{thm}[{\cite[Theorem~1.1]{Pila09a}}]\label{thm:unif}
	Let $V \subset \AB$ be a geometrically irreducible, non-special algebraic curve defined over a number field $L$.
	Write $X, Y$ for the coordinate functions on $\AB$ restricted to $V$.
	Let $d_1 = \deg X$ and $d_2 = \deg Y$.
	There exists a constant $c([L : \Q], \max\{d_1, d_2\} )$ such that
	\begin{align}\label{eq:unifdiscbd}
		\max \{ \lvert \Delta_x \rvert, \lvert \Delta_y \rvert\} \leq c([L : \Q], \max\{d_1, d_2\}) 
	\end{align} 
	for every special point $(x, y) \in V$.
\end{thm}

No effective version of Theorem~\ref{thm:unif} is known in general.
 In the special case that $V \subset \AB$ has Zariski-closure in $(\pr \times \pr)(\C)$ equal to $V \cup \{ (\infty, \infty)\}$, K\"uhne \cite[Theorem~4]{Kuhne13} proved an effective version of Theorem~\ref{thm:unif}.
Note that the class of such curves neither contains, nor is contained in, the class of curves considered in Theorem~\ref{thm:main}.
K\"uhne's proof depends crucially on the fact that $\{ (\infty, \infty)\}$ is the only point at infinity of $V$, and so does not seem generalisable to other non-special curves.

Theorem~\ref{thm:WBH} provides a strategy to prove Theorem~\ref{thm:unif} effectively for general non-special curves.
Indeed, it reduces the problem to proving, for some suitable $\eta$ and $\kappa$, the height bound $\mathrm{WBH}(\eta, \kappa)$ with an effectively computable constant $c_*$.
In Proposition~\ref{prop:wkbd}, we will show, by explicating the constant in \cite[Theorem~1.1]{Habegger10}, that, for the class of curves appearing in Theorem~\ref{thm:main}, the statement $\mathrm{WBH}(\eta, 1)$ holds for every $\eta \in (0, 1/2)$. 
Unfortunately, this falls just short of the threshold in part (2) of Theorem~\ref{thm:WBH} that is required to prove an effective version of Theorem~\ref{thm:unif} for these curves.

\subsection{The point counting approach to Andr\'e's theorem}

The proofs of Theorems~\ref{thm:main} and \ref{thm:WBH} will use an effective form of the Pila--Zannier strategy of o-minimal point counting (see \cite{Pila22} for an overview of this approach to unlikely intersections problems).
An ineffective version of this strategy was used by Pila \cite{Pila09a} to prove Theorem~\ref{thm:unif}.
The point counting part of the Pila--Zannier strategy can, in applications to Andr\'e--Oort for $\AN$, be made effective by using effective point counting results of Binyamini \cite{Binyamini19a, Binyamini24}, in place of the ineffective Pila--Wilkie theorem \cite{PilaWilkie06}.

The Pila--Zannier strategy also requires some arithmetic inputs.
In particular, to give a proof of Andr\'e's theorem, the strategy requires a lower bound for the Galois orbit of a special point $(x, y) \in V$ of the form
\begin{align}\label{eq:LGO}
	[\Q(x, y) : \Q] \gg \max \{ \lvert \Delta_x \rvert, \lvert \Delta_y \rvert\}^\theta
	\end{align}
for some $\theta > 0$.
Pila \cite[p.~2497]{Pila09a} deduces this from the Landau--Siegel bound \cite{Landau35,Siegel35} for the class number of an imaginary quadratic field:
\begin{align}\label{eq:LS}
	 \cl(D) \gg \lvert D \rvert^{1/2 - \epsilon},
	 \end{align}
where $\cl(D)$ denotes the class number of the imaginary quadratic field of discriminant $D$.
Unfortunately, the bound \eqref{eq:LS} is ineffective.
The only known effective bounds for the class number, due to Goldfeld \cite{Goldfeld76} and Gross and Zagier \cite{GrossZagier86}, are not strong enough to imply a bound \eqref{eq:LGO}. 

Tatuzawa \cite{Tatuzawa51} proved an effective version of \eqref{eq:LS} that holds for the discriminant of every imaginary quadratic field, apart from at most one exceptional field $K_*$ (depending on $\epsilon$).
Binyamini applied this to prove a result \cite[Corollary~1]{Binyamini19} which has as a special case an effective (and uniform) version of Andr\'e's theorem for special points $(x, y)$ with $(\sqrt{\Delta_x}, \sqrt{\Delta_y}) \notin K_*^2$.

Unfortunately, Tatuzawa's result does not give any information about the field $K_*$.
It is therefore necessary to find a different way of controlling special points with $(\sqrt{\Delta_x}, \sqrt{\Delta_y}) \in K_*^2$.
In the special case that $V$ is defined by a polynomial $F \in \alg[X, Y]$ with $\deg_X F = \deg_Y F = \deg F$, Binyamini \cite[Corollary~3]{Binyamini19} was able to do this via Archimedean estimates for special points, but in a way that depends essentially on the particular form of $F$.

For some curves $V$, one may prove an effective version of \eqref{eq:LGO} directly.
Papas \cite[Theorem~1.5]{Papas26} did this for a certain class of curves by using Andr\'e's $G$-functions method, which was originally applied to unlikely intersections in $\AN$ by Daw and Orr \cite{DawOrr25}.
Papas proved \cite[Theorem~2.17]{Papas26} that if $V$ satisfies a suitable boundary condition, then there exist constants $c_1, c_2$ such that every special point $(x, y) \in V$ satisfies
\begin{align}\label{eq:PapHt}
	\max \{ h(x), h(y)\} \leq c_1 [\Q(x, y) : \Q]^{c_2}.
\end{align}
The constants $c_1, c_2$ are effectively computable in terms of some quantities related to the $G$-functions used \cite[Remark~6.1]{Papas26}.
From \eqref{eq:PapHt}, Papas then deduces a bound of the form \eqref{eq:LGO} by using endomorphism estimates of Masser and W\"ustholz \cite{MasserWustholz94}.
This Galois bound is then used to prove \cite[Corollary~1.7]{Papas26} via an effective version of the Pila--Zannier strategy.

A drawback of the $G$-functions approach is that it requires the curve $V$ to satisfy a suitable boundary condition.
(Papas was also able to obtain a partial result \cite[Theorem~6.4]{Papas25} without a boundary condition, but with a restriction on the special points considered instead.)
It is also not completely straightforward to calculate the constants $c_1$ and $c_2$ in \eqref{eq:PapHt}.

\subsection{Strategy of proof}

To prove Theorems~\ref{thm:main} and \ref{thm:WBH}, we will use a new way of controlling special points $(x, y) \in V$ with $(\sqrt{\Delta_x}, \sqrt{\Delta_y}) \in K_*^2$.
The advantages of this new approach are that it leads to a good dependence on the height of the curve $V$ and requires only that $V$ satisfies a height bound along the lines of the one formulated in Definition~\ref{def:WBH}.
It therefore seems to offer a more promising approach than \cite[Corollary~1.7]{Papas26} and \cite[Corollary~3]{Binyamini19} to proving the uniform Theorem~\ref{thm:unif} effectively.

Moreover, our method also gives a way of deducing Andr\'e--Oort results directly from height bounds for special points of the form of \eqref{eq:PapHt}.
This method differs from the usual approach of applying Masser--W\"ustholz estimates and point counting. 
We explain this in Section~\ref{sec:G}.

We now give an outline of the proof of Theorem~\ref{thm:main}. A special point $(x, y) \in V$ with $(\sqrt{\Delta_x}, \sqrt{\Delta_y}) \in K_*^2$ necessarily lies in the intersection of $V$ with a modular curve $\mathbb{V}(\Phi_N(X, Y))$. 
We take $N \in \Z_{> 0}$ to be the least such $N$. 
Under the assumptions on $V$ in Theorem~\ref{thm:main}, a result of Habegger \cite[Theorem~1.1]{Habegger10}
implies that
\[\max \{h(x), h(y) \} \ll \log N.\]
Isogeny estimates 
of Gaudron and R\'emond \cite{GaudronRemond14} then give that
\[ [\Q(x, y) : \Q] \gg N^{1/3}.\]
On the other hand, if the special point $(x, y)$ is such that 
\[ [\Q(x) : \Q] \ll  \lvert \Delta_x \rvert^{1/2 - \epsilon} \mbox{ and } [\Q(y) : \Q] \ll  \lvert \Delta_y \rvert^{1/2 - \epsilon},\]
then, via Archimedean estimates for the $j$-function (Proposition~\ref{prop:htsinglow}),
\[ \max \{h(x), h(y)\} \gg \max \{ \lvert \Delta_x \rvert, \lvert \Delta_y \rvert\}^{\epsilon}.\]

We now have upper and lower bounds for $\max \{h(x), h(y)\}$ and $[\Q(x, y) : \Q]$, which are incompatible with each other if $\max \{ \lvert \Delta_x \rvert, \lvert \Delta_y \rvert\}$ is sufficiently large.
The bound on $\max \{ \lvert \Delta_x \rvert, \lvert \Delta_y \rvert\}$ we thereby obtain (Theorem~\ref{thm:except}) is completely explicit and sharp enough to allow one to rule out the existence of any such special points on many curves $V$, see Remark~\ref{rmk:sharp}.

The proof of part (1) in Theorem~\ref{thm:WBH} is similar, but using the assumption of $\mathrm{WBH}(\eta, \kappa)$ in place of Habegger's result.
Part (2) then follows from (1) via uniform bounds on the number of points of small height on an affine curve, as used by K\"uhne \cite[\S3]{Kuhne13}.

There is a subtlety to note in the above outline. 
 Tatuzawa's result implies that if $(x, y)$ is a special point such that $x$ and $y$ have abnormally small degree, then $(\sqrt{\Delta_x}, \sqrt{\Delta_y}) \in K_*^2$.
 It does not imply though that all special points with $(\sqrt{\Delta_x}, \sqrt{\Delta_y}) \in K_*^2$ have $x$ and $y$ of abnormally small degree.
 Those special points where $x$ and $y$ do not both have abnormally small degree though may be bounded by effective point counting.
 
 K\"uhne's \cite{Kuhne12} effective proof of Andr\'e's theorem does not use point counting.
 He too was able, by using Tatuzawa's version of \eqref{eq:LS} with $\epsilon < 1/10$, to prove a uniform, effective bound \cite[Corollary~2]{Kuhne13} on special points $(x, y) \in V$ with $(\sqrt{\Delta_x}, \sqrt{\Delta_y}) \notin K_*^2$.
 
 We will use this result of K\"uhne to give a second proof of the bound \eqref{eq:htbd} in Theorem~\ref{thm:main}.
 In particular, one still obtains from this proof the linear dependence on $h(V)$ in \eqref{eq:htbd}. 
 This second proof has the advantage that it does not use any point counting and is therefore more conducive to computing the constant $c_1$ in \eqref{eq:htbd} explicitly.
 The disadvantage is that the version of the bound \eqref{eq:discbd} one obtains in this way has the exponent $24 + \delta$ rather than $2 + \delta$, and so is weaker than \cite[Theorem~2]{Kuhne12}.

	\section{Preliminaries}

\subsection{Heights}\label{sec:hts}

Denote by $h(x)$ the absolute logarithmic height of an algebraic number $x$, as defined in \cite[\S1.5]{BombieriGubler06}.
When we refer to the ``height'', we mean this height.

Given a geometrically irreducible curve $V \subset \AB$ defined over a number field $L$, we write $h(V)$ for the height of the curve $V$. 
This is defined to be the projective logarithmic height of the vector whose entries are the non-zero coefficients of a minimal defining polynomial for $V$.

Since $x, y$ will always denote algebraic numbers and $V$ will always denote an algebraic curve, we trust that there will be no confusion between the notation $h(x)$ (respectively $h(y)$) and $h(V)$.

Given an elliptic curve $E / \alg$, we write $\falt(E)$ for the stable Faltings height of $E$. 
We adopt the same normalisation for $\falt(E)$ as in \cite{GaudronRemond14, Pazuki19}, so that we may quote results from those papers without changing the constants.

\subsection{Singular moduli}

Let $j \colon \h \to \C$ denote the modular $j$-function, where $\h$ denotes the complex upper half-plane. 
The $j$-function is surjective and invariant under the action of $\SL$ on $\h$.
A complex number $x = j(\tau)$ is a singular modulus if and only if $[\Q(\tau) : \Q]  = 2$.

Let $\tau \in \h$ be such that
\[ a \tau^2 + b \tau + c = 0\]
for some $a, b, c \in \Z$, not all zero, such that $\gcd(a, b, c) = 1$.  
Then $x = j(\tau)$ is a singular modulus, and the discriminant $\Delta_x$ of $x$ satisfies
\[  \Delta_x = b^2 - 4 ac.\]
The \textit{fundamental discriminant} $D_x$ of $x = j(\tau)$ is the discriminant of the imaginary quadratic field $\Q(\tau)$. 
One has that $\Delta_x = f^2 D_x$ for some $f \in \Z_{> 0}$.
The integer $f$ is the conductor of the endomorphism order of an elliptic curve $E_x$ with $j$-invariant $x$, i.e~$\End(E_x) \cong \Z + f \mathcal{O}_K$, where $K = \Q(\sqrt{\Delta_x})$.

The singular moduli of a given discriminant form a complete set of Galois conjugates over $\Q$, see e.g.~\cite[Proposition~13.2]{Cox22}.
Hence, if $x$ is a singular modulus, then
\[ [\Q(x) : \Q] = \cl(\Delta_x),\]
where $\cl(\Delta_x)$ denotes\footnote{We use this notation rather than introducing a third meaning for $h(\cdot)$.} the class number of the imaginary quadratic order of discriminant $\Delta_x$.
The results of Goldfeld \cite{Goldfeld76} and Gross--Zagier \cite{GrossZagier86} give an effective lower bound for for $\cl(\Delta_x)$, which is however qualitatively much weaker than the one given by Landau and Siegel \cite{Landau35, Siegel35}.

\begin{prop}\label{prop:Gold}
	Let $x$ be a singular modulus. 
	Then
	\[ [\Q(x) : \Q] \geq \frac{1}{\sqrt{168000}} (\log \lvert \Delta_x \rvert)^{1/2}\]
\end{prop}

\begin{proof}
Let $D_x$ be the fundamental discriminant of $x$ and let $f \in \Z_{> 0}$ be such that $\Delta_x = f^2 D_x$.
	It follows from Oesterl\'e's explicit version \cite[Th\'eor\`eme~1 \& \S5.1]{Oesterle85} of the result of Goldfeld and Gross--Zagier that
	\[\cl(D_x)^2 \geq \frac{\log \lvert D_x \rvert}{42000},\] 
	see \cite[Proposition~2.7]{Fowler25} for the details.
	Let $\varphi(\cdot)$ denote Euler's totient function.
Then \cite[Theorem~7.24]{Cox22} implies that
	\[ \cl(\Delta)^2 \geq \varphi(f)^2 \cl(D)^2.\]
	Hence, the elementary bound $\varphi(f) \geq \sqrt{f/2}$ yields that
	\begin{align*}
		\cl(\Delta)^2 &\geq \frac{f \log \lvert D \rvert}{84000} \geq \frac{\log \lvert D \rvert + (f-1)\log \lvert D \rvert}{84000}\\
		&\geq \frac{\log \lvert D \rvert + 2 \log f}{168000},
	\end{align*} 
since $\lvert D \rvert \geq 3$ implies $(f-1)\log \lvert D \rvert \geq f - 1 \geq \log f$ for every $f \in \Z_{> 0}$.
\end{proof}

Weaker versions of the following bound are well-known, see e.g.~\cite[Lemme~1(iii)]{FaisantPhilibert87}, \cite[Lemma~1]{BiluMasserZannier13}, and \cite[Lemma~2.5]{Pazuki19}.
The constant $744$ is sharp, as may be seen by taking $z = -1/2 + iy$ and letting $y \to \infty$.

\begin{lem}\label{lem:jbd}
	Let $z \in \h$. Then 
	\[ \lvert j(z) \rvert > e^{2 \pi \im z} - 744.\]
\end{lem}

\begin{proof}
	Since $j(z)$ is invariant under $z \mapsto z + 1$ and $e^{\sqrt{3} \pi} < 744$, we may assume that $-1/2 \leq \re z \leq 1/2$ and $\im z \geq \sqrt{3}/2$.
	Hence, by a result of Diaz and Philibert \cite[Theorem~(i)]{DiazPhilibert89},
	\[ \lvert j(z) \rvert \geq \left \lvert j\left(-\frac{1}{2} + i \im z\right) \right\rvert = -j\left(-\frac{1}{2} + i \im z\right),\]
	where the equality follows from the fact that the $j$-function takes non-positive real values on the line segment $\{ -1/2 + i y : y \geq \sqrt{3}/2\}$.
	The $q$-expansion of the $j$-function then gives that
	\[ \lvert j(z) \rvert - e^{2 \pi \im z} \geq -744 - \sum_{n \geq 1} c_n (-e^{-2 \pi \im z})^n, \]
	where the $c_n$ are the coefficients in the $q$-expansion of the $j$-function.
	
	It therefore suffices to show that the sum
	\begin{align}\label{eq:sum} 
		\sum_{n \geq 1} (-1)^{n-1} c_n (e^{- 2 \pi \im z})^n
		\end{align}
	is positive.
	For every $n \in \Z_{> 0}$, it follows from \cite[Theorem~1.1]{BrisebarrePhilibert05}  that
	\[  0.91 \frac{e^{4 \pi \sqrt{n}}}{\sqrt{2} n^{3/4}}  \leq c_n \leq 1.03 \frac{e^{4 \pi \sqrt{n}}}{\sqrt{2} n^{3/4}}.\]
	These bounds and and the fact $\im z \geq \sqrt{3}/2$ together imply that
	\begin{align*}
		\frac{c_{n+1} (e^{-2 \pi \im z})^{n+1}}{c_n (e^{-2 \pi \im z})^n} &\leq \frac{1.03}{0.91} \frac{e^{4 \pi (\sqrt{n+1} - \sqrt{n})}}{(1+ \frac{1}{n})^{3/4}}e^{ -\pi \sqrt{3}}\\
		&\leq 1.14 e^{4 \pi  (\sqrt{2} - 1)} e^{ -\pi \sqrt{3}}.
	\end{align*}
Since $1.14 e^{4 \pi  (\sqrt{2} - 1)} e^{ -\pi \sqrt{3}} < 0.91$, the sum in \eqref{eq:sum} is positive as required.
\end{proof}

This result implies a lower bound for the height of a singular modulus. 

\begin{prop}\label{prop:htsinglow}
	Let $x$ be a non-zero singular modulus. Then
	\[ h(x) \geq \frac{3\lvert \Delta_x \rvert^{1/2} }{[\Q(x) : \Q]}.\]
\end{prop}

\begin{proof}
	Let $k \in \{0, 1\}$ be such that $k \equiv \Delta_x \bmod 2$ and set
	\[ \tau = \frac{-k + \lvert \Delta_x \rvert^{1/2} i}{2}.\]
	Then $j(\tau)$ is a singular modulus of discriminant $\Delta_x$.
	So $x$ is conjugate over $\Q$ to $j(\tau)$.  
	Thus, by basic properties of the height (see e.g.~\cite[(1.8)]{BombieriGubler06}),
	\[ h(x) \geq \frac{\log \lvert j(\tau) \rvert}{[\Q(x) : \Q]}.\]
	By Lemma~\ref{lem:jbd},
	\[ \lvert j(\tau) \rvert \geq e^{\pi \lvert \Delta_x \rvert^{1/2}} - 744.\]
	Since $x \neq 0$, we have that $\Delta_x \neq -3$. 
	For $\Delta_x \in \{-4, -7\}$, the inequality is checked directly.
	Assume then that $\lvert \Delta_x \rvert \geq 8$. 
	It now suffices to observe that
	\[ \log (1 - 744e^{-\pi \lvert \Delta_x \rvert^{1/2}}) \geq \log (1 - 744e^{-\pi \sqrt{8}})  > -0.11. \qedhere\]
\end{proof}

	\subsection{Modular polynomials}

We recall some basic facts about modular polynomials and their relation to singular moduli and isogenies. 

\begin{prop}[{\cite[Ch.5 \S3]{Lang87}}]\label{prop:modpoly}
	Let $x, y \in \C$. 
	\begin{enumerate}
		\item Let $N \in \Z_{> 0}$ and $E_x, E_y$ be elliptic curves with $j$-invariants $x, y$ respectively. 
		Then $\Phi_N(x, y) = 0$ if and only if there exists an isogeny $\varphi \colon E_x \to E_y$ such that $\ker \varphi \cong \Z / N \Z$, i.e. $\varphi$ is a cyclic $N$-isogeny.
		\item Suppose that $\Phi_N(x, y) = 0$ for some $N \in \Z_{> 0}$. Then $x$ is a singular modulus if and only if $y$ is a singular modulus.
		\item Suppose that $x, y$ are singular moduli with fundamental discriminants $D_x, D_y$ respectively. Then $\Phi_N(x, y) = 0$ for some $N \in \Z_{> 0}$ if and only if $D_x = D_y$.
	\end{enumerate}
\end{prop}

The following isogeny estimate allows us to obtain lower bounds for the Galois orbit of a special point from an upper bound for its height.

\begin{prop}\label{prop:isog}
	Let $x, y \in \alg$ and set $d = [\Q(x, y) : \Q]$. Let $E_x, E_y / \alg$ be elliptic curves with $j$-invariants $x, y$ respectively. If $\Phi_M(x, y) = 0$ for some $M \in \Z_{>0}$, then there exists $N \in \Z_{>0}$ such that $\Phi_N(x, y) = 0$ and
	\[ N \leq 10^{13} \max \{1, \falt(E_x), \log d\}^2 d^2.\]
\end{prop}

\begin{proof}
	This follows immediately from the explicit isogeny estimate of Gaudron and R\'emond \cite[Th\'eor\`eme~1.4]{GaudronRemond14}, since every isogeny of minimal degree is cyclic \cite[Lemma~7.2]{GaudronRemond14}.
\end{proof}

	\section{Weakly bounded height}
	
	To prove Theorem~\ref{thm:main}, we will make use of a result of Habegger \cite[Theorem~1.1]{Habegger10} that gives a height bound for points in the intersection between a curve
	 satisfying the hypotheses of Theorem~\ref{thm:main} and a modular curve. 
	Our aim in this section is to make explicit how the constant in Habegger's result depends on the height of the curve in question. 
	We start by recalling another result of Habegger \cite{Habegger17}.
	\begin{prop}[{\cite[Theorem~1]{Habegger17}}]\label{prop:quasi} 
		Let $V \subset \AB$ be an irreducible algebraic curve defined over $\alg$. 
		Write $X, Y$ for the coordinate functions on $\AB$ restricted to $V$.
		Suppose that $d_1 = \deg X > 0$ and $d_2 = \deg V > 0$.
		Let $L_{d_1, d_2} = \log(2^{\min \{d_1, d_2\}} (d_1+1)(d_2+1))$.
		If $(x, y) \in V(\alg)$, then
		\begin{align}\label{eq:Hab}
			 \left\lvert \frac{h(x)}{d_1} - \frac{h(y)}{d_2} \right\rvert \leq 5(L_{d_1, d_2} +  h(V))^{1/2} \max \left\{ \frac{h(x)}{d_1}, \frac{h(y)}{d_2}\right\}^{1/2}.
			 \end{align}
	\end{prop}

\begin{remark}\label{rmk:exp}
	The constant $5(L_{d_1, d_2} + h(V))^{1/2}$ in \eqref{eq:Hab} cannot in general be replaced by a constant which is $o(h(V)^{1/2})$.
	To see this, let $d_1, d_2 \in \Z_{> 0}$ with $(d_1, d_2) = 1$ and consider, for $n \in \Z_{> 0}$, the curve $V_n$ defined by $F_n(X, Y) = X^{d_2} - n Y^{d_1}$ with the point $(n^{1/d_2}, 1) \in V_n$.
\end{remark}

We may now state the following result, which is the requisite explicit form of \cite[Theorem~1.1]{Habegger10}. The proof is exactly that of Habegger.

		\begin{prop}\label{prop:wkbd}
				Let $V \subset \AB$ be an irreducible algebraic curve defined over $\alg$. 
			Write $X, Y$ for the coordinate functions on $\AB$ restricted to $V$.
			Suppose that $d_1 = \deg X > 0$ and $d_2 = \deg V > 0$ and, further, $d_1 \neq d_2$.  
			Let $L_{d_1, d_2}$ be as in Proposition~\ref{prop:quasi}.
			If $(x, y) \in V(\alg)$ is such that 
		\[\max \left\{ \frac{h(x)}{d_1}, \frac{h(y)}{d_2}\right\} \geq 26 \max \{d_1, d_2\}^2 (L_{d_1, d_2} + h(V))\]
		and $N \in \Z_{>0}$ is such that $\Phi_N(x, y) = 0$, then
		\[ \max \left\{ \frac{h(x)}{d_1}, \frac{h(y)}{d_2}\right\} \leq 474 + 618 \log N.\]
	\end{prop}

	\begin{proof}
	Let 
		\[ B = \max \left\{ \frac{h(x)}{d_1}, \frac{h(y)}{d_2}\right\},\]
		and suppose that
		\begin{align}\label{eq:M}
			 B \geq 26 \max \{d_1, d_2\}^2 (L_{d_1, d_2} + h(V)).
			 \end{align}
		Since $(x, y) \in V(\alg)$, Proposition~\ref{prop:quasi} implies that
		\[ \left\lvert \frac{h(x)}{d_1} - \frac{h(y)}{d_2} \right\rvert \leq 5(L_{d_1, d_2} +  h(V))^{1/2} B^{1/2}.\]
		Hence, the assumption \eqref{eq:M} yields that
		 	\begin{align}\label{eq:M1}
		 		 \left\lvert \frac{h(x)}{d_1} - \frac{h(y)}{d_2} \right\rvert \leq \frac{5B}{\sqrt{26} \max \{d_1, d_2\}}.
		 		 \end{align}
		 	
		 Now observe that,	since $d_1 \neq d_2$,
		 	\begin{align}\label{eq:M2}
		 		\frac{h(x)}{d_1} &\leq \frac{h(x)}{d_1} \lvert d_1 - d_2 \rvert \nonumber\\
		 		&\leq d_2 \left\lvert \frac{h(x)}{d_1} - \frac{h(y)}{d_2} \right\rvert + \lvert h(y) - h(x) \rvert,
		 	\end{align}
	 	and, similarly,
	 	\begin{align}\label{eq:M3}
	 		\frac{h(y)}{d_2} \leq d_1  \left\lvert \frac{h(y)}{d_2} - \frac{h(x)}{d_1} \right\rvert + \lvert h(x) - h(y) \rvert.
	 	\end{align}
 From \eqref{eq:M1}, \eqref{eq:M2}, and \eqref{eq:M3}, we thus obtain that
 	\[ \left(1 - \frac{5}{\sqrt{26}} \right) B \leq  \lvert h(x) - h(y) \rvert.\]
 			Finally, note that if $\Phi_N(x, y) = 0$, then \cite[Theorem~1.1]{Pazuki19} implies that
 			 \[\lvert h(x) - h(y) \rvert \leq 9.204 + 12 \log N. \qedhere \]
 			 	 \end{proof}
 		 	 
 		 	 Proposition~\ref{prop:wkbd} implies that special points $(x, y) \in V \cap \mathbb{V}(\Phi_N)$ satisfy
 		 	 \[ \max \{h(x), h(y)\} \leq c_1 h(V) + c_2 \log N,\]
 		 	 where $c_1, c_2$ depend only on $\max \{d_1, d_2\}$ and $[L : \Q]$.
 		 	 This is much stronger in its dependence on $N$ than the condition $\mathrm{WBH}(\eta, \kappa)$ in Theorem~\ref{thm:WBH}.
 		 	 Nonetheless, Proposition~\ref{prop:wkbd} is still not sufficient to prove a uniform version of Andr\'e's theorem, because of the linear dependence on $h(V)$. 
 		 	 Remark~\ref{rmk:exp} shows that obtaining a better than linear dependence on $h(V)$ will require a different approach than using Proposition~\ref{prop:quasi} alone.

	\section{Tatuzawa's theorem and $\epsilon$-exceptional special points}
	
	\subsection{Tatuzawa's theorem}
	
	Tatuzawa \cite{Tatuzawa51} proved that Siegel's \cite{Siegel35} ineffective lower bound for the class number of an imaginary quadratic field can be made effective, apart from a single possible exceptional field.
	
	\begin{prop}[{\cite[Theorem~1]{Tatuzawa51}}]\label{prop:Tat}
		Let $\epsilon \in (0, 1/2)$. Then
		\begin{align}\label{eq:Tat}
			 \cl(D) \geq \frac{\epsilon}{10 \pi} \lvert D \rvert^{1/2 - \epsilon}
			 \end{align}
		for every fundamental discriminant $D$ of an imaginary quadratic field, apart from at most one exception. 	
	\end{prop}

\begin{definition}
	Let $\epsilon \in (0, 1/2)$. Denote by $D_\epsilon$ the unique (for the given $\epsilon$) $D$ for which \eqref{eq:Tat} fails, if it exists. 
\end{definition}

Tatuzawa's theorem extends to non-fundamental discriminants via classical bounds for the totient function $\varphi(\cdot)$.

\begin{lem}\label{lem:euler}
	Let $\delta \in (0, 1)$. Then $\varphi(n) \geq  2^{-1/\delta} n^{1 - \delta}$ for every $n \in \Z_{>0}$.
\end{lem}

\begin{proof}
	Observe that
	\begin{align*}
		\varphi(n) &= n \prod_{p \mid n} \left( 1- \frac{1}{p}\right) \geq n \prod_{\substack{k \geq 2 \\ k^\delta \leq 2}} \left( 1- \frac{1}{k}\right) \prod_{\substack{p \mid n \\ p^\delta > 2}} \left( 1- \frac{1}{p}\right) \geq n 2^{-1/\delta} \prod_{p \mid n} p^{-\delta},
	\end{align*}
since $p^\delta > 2$ implies that $p^{-\delta} < 1/2 < 1 - 1/p$.
\end{proof}

	\begin{prop}\label{prop:TatNon}
		Let $\epsilon \in (0, 1/2)$. Let $D \neq D_\epsilon$ be a fundamental discriminant of an imaginary quadratic field. Let $f \in \Z_{>0}$ and $\Delta = f^2 D$. Then
		\[ \cl(\Delta) \geq \frac{\epsilon}{30 \pi (\sqrt{2})^{1/ \epsilon}} \lvert \Delta \rvert^{1/2 - \epsilon}.\] 
	\end{prop}

\begin{proof}
	Let $K = \Q(\sqrt{D})$ and denote by $\mathcal{O}_K$ its ring of integers. Let $\mathcal{O}_f$ denote the order of conductor $f$ in $\mathcal{O}_K$. Then, by \cite[Theorem~7.24]{Cox22},
	\[ \cl(\Delta) = \frac{\cl(D)}{[\mathcal{O}_K^\times : \mathcal{O}_f^\times]} f \prod_{p \mid f} \left (1- \left(\frac{D}{p}\right)\frac{1}{p}\right),\]
	where here and hereafter $(D/p)$ denotes the Kronecker symbol. Since $[\mathcal{O}_K^\times : \mathcal{O}_f^\times] \leq 3$ and $(D / p) \in \{-1, 0, 1\}$, we have that
	\[ \cl(\Delta) \geq \frac{\cl(D)}{3} \varphi(f).\]
	The result thus follows from Proposition~\ref{prop:Tat} and Lemma~\ref{lem:euler}.
\end{proof}

\subsection{{$\epsilon$}-exceptional special points}

\begin{definition}\label{def:except}
	Let $\epsilon \in (0, 1/2)$. Let $x, y$ be singular moduli of respective fundamental discriminants $D_x, D_y$. 
	Let $f, g \in \Z_{>0}$ be such that $\Delta_x = f^2 D_x$ and $\Delta_y = g^2 D_y$. 
	Let $l = \lcm(f, g)$.
	The special point $(x, y)$ is called $\epsilon$-\textbf{exceptional} if $D_x = D_y$ and
	\[ \cl(l^2 D_x) < \frac{\epsilon}{30 \pi (\sqrt{2})^{1/ \epsilon}} \lvert l^2 D_x \rvert^{1/2 - \epsilon}.\]
\end{definition}

\begin{remark}\label{rmk:except}
	Observe that if the special point $(x, y)$ is $\epsilon$-exceptional, then $D_x = D_y = D_\epsilon$ by Proposition~\ref{prop:TatNon}. 
	Note however that a special point $(x, y)$ with $D_x = D_y = D_\epsilon$ is not necessarily $\epsilon$-exceptional. 
	Indeed, if $p$ is prime, then \cite[Theorem~7.24]{Cox22} implies that
	\[ \cl(p^2 D) \geq \frac{p-1}{3} \cl(D),\]
	but $(p-1)/p^{1 - 2 \epsilon} \to \infty$ as $p \to \infty$.
\end{remark}

The key property of $\epsilon$-exceptional special points $(x, y)$ is that both $\Delta_x$ and $\Delta_y$ have small class numbers.

\begin{prop}\label{prop:except}
	Let $\epsilon \in (0, 1/2)$. 
	Let $(x, y)$ be an $\epsilon$-exceptional special point. 
	Then
	\[ \cl(\Delta_x) < \frac{\epsilon}{10 \pi} \lvert \Delta_x \rvert^{1/2 - \epsilon} \mbox{ and } \cl(\Delta_y) < \frac{\epsilon}{10 \pi} \lvert \Delta_y \rvert^{1/2 - \epsilon}.\]
\end{prop}

\begin{proof}
	Denote by $D$ the common fundamental discriminant of $x$ and $y$.
	Let $f, g \in \Z_{>0}$ be such that $\Delta_x = f^2 D$ and $\Delta_y = g^2 D$.
	Let $l = \lcm(f, g)$. 
	Since the special point $(x, y)$ is $\epsilon$-exceptional, 
	\begin{align}\label{eq:smallclass}
		 \cl(l^2 D) < \frac{\epsilon}{30 \pi (\sqrt{2})^{1/ \epsilon}} \lvert l^2 D \rvert^{1/2 - \epsilon}.
		 \end{align}
	Let $\mathcal{O}_f, \mathcal{O}_l$ denote the imaginary quadratic orders of discriminant $f^2 D$ and $l^2 D$ respectively. 
	By \cite[Corollary~7.28]{Cox22},
	\[\cl(l^2 D) = \frac{\cl(\Delta_x)}{ [\mathcal{O}_f^\times : \mathcal{O}_l^\times]} \frac{l}{f} \prod_{\substack{p \mid \frac{l}{f}}} \left(1- \left(\frac{f^2 D}{p}\right)\frac{1}{p}\right).\]
	Since $[\mathcal{O}_f^\times : \mathcal{O}_l^\times] \leq 3$ and the Kronecker symbol takes values in $\{-1, 0, 1\}$, we obtain that
	\[\frac{1}{3} \cl(f^2 D) \varphi\left(\frac{l}{f}\right) \leq \cl(l^2 D).\]
Therefore, Lemma~\ref{lem:euler} and \eqref{eq:smallclass} 	together imply that
	\[ \cl(\Delta_x) < \frac{\epsilon}{10 \pi} \lvert f^2 D \rvert^{1/2 - \epsilon}.  \]
	The situation for $\cl(\Delta_y)$ is analogous.
\end{proof}

Proposition~\ref{prop:except} gives an upper bound for the degrees of $x$ and $y$. This yields a lower bound for the heights of $x$ and $y$, thanks to Proposition~\ref{prop:htsinglow}.

\begin{prop}\label{prop:htexcept}
	Let $\epsilon \in (0, 1/2)$. 
	Let $(x, y)$ be an $\epsilon$-exceptional special point. 
	Then
	\[ h(x) \geq \frac{30 \pi}{\epsilon} \lvert \Delta_x \rvert^{ \epsilon} \mbox{ and } h(y) \geq \frac{30 \pi}{\epsilon} \lvert \Delta_y \rvert^{ \epsilon}. \]
\end{prop}

\begin{proof}
	This is an immediate consequence of Propositions~\ref{prop:htsinglow} and \ref{prop:except}. 
\end{proof}

\subsection{Non-{$\epsilon$}-exceptional special points}

If a special point $(x, y)$ with $D_x = D_y$ is not $\epsilon$-exceptional, then its Galois orbit is ``large'' (i.e.~it satisfies \eqref{eq:LGO}).
This is true even if $D_x = D_y = D_\epsilon$.

\begin{prop}\label{prop:nonexcept}
	Let $\epsilon \in (0, 1/2)$. 
	Let $x, y$ be singular moduli of the same fundamental discriminant $D$. 
	If $(x, y)$ is not $\epsilon$-exceptional, then
	\[ [\Q(x, y) : \Q] \geq \frac{\epsilon}{90 \pi (\sqrt{2})^{1/ \epsilon}} \max \{\lvert \Delta_x \rvert, \lvert \Delta_y \rvert\}^{1/2 - \epsilon}.\]
\end{prop}

\begin{proof}
	Let $f, g \in \Z_{>0}$ be such that $\Delta_x = f^2 D$ and $\Delta_y = g^2 D$. 
	Let $l = \lcm(f, g)$. 
	Let $K = \Q(\sqrt{D})$. 
	By \cite[Proposition~3.1]{AllombertBiluMadariaga15}, the field $K(x, y)$ is a subfield of degree at most $3$ of the ring class field of discriminant $l^2 D$. 
	The ring class field of discriminant $l^2 D$ is itself an extension of $K$ of degree $\cl(l^2 D)$. Hence,
	\[  [K(x, y) : K] \geq \frac{\cl(l^2 D)}{3}.\]
	Since the special point $(x, y)$ is not $\epsilon$-exceptional, by definition we have that
	\[ \cl(l^2 D) \geq \frac{\epsilon}{30 \pi (\sqrt{2})^{1/ \epsilon}} \lvert l^2 D \rvert^{1/2 - \epsilon}.\]
	Note that $\max \{\lvert \Delta_x \rvert, \lvert \Delta_y \rvert\} \leq \lvert l^2 D \rvert$.  We thus obtain that
	\begin{align*} 
		[\Q(x, y) : \Q] &\geq   \frac{\epsilon}{90 \pi (\sqrt{2})^{1/ \epsilon}} \max \{\lvert \Delta_x \rvert, \lvert \Delta_y \rvert\}^{1/2 - \epsilon}. \qedhere
	\end{align*}
\end{proof}

In fact, we may prove a similar lower bound for the Galois orbit of one of $x$ and $y$, provided $\epsilon$ is sufficiently small. This requires an easy lemma about the Dedekind $\psi$-function $\psi(\cdot)$.

\begin{lem}\label{lem:DedPsi}
	Let $\delta > 0$. 
	Then $\psi(n) \leq n^{1+ \delta}/ \delta$ for every $n \in \Z_{\geq 2}$.
\end{lem}

\begin{proof}
	Let $f_n(t) = n^{1+t}/t$. 
	Then, for every $\delta > 0$,
	\[ f_n(\delta) \geq f_n\left(\frac{1}{\log n}\right) = e n \log n.\]
	In particular, $f_n(\delta) \geq \psi(n)$ for $2 \leq n \leq 6$.
	For $n \geq 7$, note that
	\begin{align*} 
		\frac{\psi(n)}{n}  \leq \prod_{p \mid n} \left(1 - \frac{1}{p}\right)^{-1} = \frac{n}{\varphi(n)} &\leq e^\gamma \log \log n + \frac{2.51}{\log \log n},
	\end{align*}
by \cite[Theorem~15]{RosserSchoenfeld62}, and the right-hand side is $\leq e \log n$ for $n \geq 7$.
\end{proof}

\begin{prop}\label{prop:degnonexcept}
	Let $\epsilon \in (0, 1/5)$. 
	Let $x, y$ be singular moduli of the same fundamental discriminant $D$.
	Let $f, g \in \Z_{> 0}$ be such that $\Delta_x = f^2 D$ and $\Delta_y = g^2 D$.
	Suppose that the special point $(x, y)$ is not $\epsilon$-exceptional. 
	Then
	\[ \cl\left(\max \{f, g \}^2 D\right) \geq  \frac{\epsilon^2}{30 \pi (\sqrt{2})^{1/ \epsilon}} \left\lvert \max \{f, g \}^2 D \right\rvert^{1/2 - 5 \epsilon/2}.\]
\end{prop}

\begin{proof} 
	Let $l = \lcm(f, g)$. 
	Without loss of generality, assume that $f \geq g$. 
		In particular, $l \leq f^2$ and $  \max \{f, g \}^2 D = \Delta_x$. 
	Let $\mathcal{O}_f, \mathcal{O}_l$ denote the imaginary quadratic orders of discriminant $f^2 D$ and $l^2 D$ respectively. 
	By \cite[Corollary~7.28]{Cox22},
	\[\cl(l^2 D) = \frac{\cl(f^2 D)}{ [\mathcal{O}_f^\times : \mathcal{O}_l^\times]} \frac{l}{f} \prod_{\substack{p \mid \frac{l}{f}}} \left(1- \left(\frac{f^2 D}{p}\right)\frac{1}{p}\right).\]
	Since the Kronecker symbol takes values in $\{-1, 0, 1\}$, we obtain that
	\[ \cl(l^2 D) \leq \cl(f^2 D) \psi\left(\frac{l}{f}\right),\]
	where $\psi(\cdot)$ denotes the Dedekind $\psi$-function. Lemma~\ref{lem:DedPsi} implies that
	\[ \cl(l^2 D) \leq \frac{1}{\epsilon} \left(\frac{l}{f} \right)^{1 + \epsilon} \cl(f^2 D).\]
	
	Since the special point $(x, y)$ is not $\epsilon$-exceptional, 
	\[ \cl(l^2 D) \geq \frac{\epsilon}{30 \pi (\sqrt{2})^{1/ \epsilon}} \lvert l^2 D \rvert^{1/2 - \epsilon}.\]
	Therefore,
	\begin{align*} 
		\cl(f^2 D) &\geq \frac{\epsilon^2 }{30 \pi (\sqrt{2})^{1/ \epsilon}} \left( \frac{f}{l} \right)^{1 + \epsilon} \lvert l^2 D \rvert^{1/2 - \epsilon}\\
		&\geq \frac{\epsilon^2}{30 \pi (\sqrt{2})^{1/ \epsilon}}  \frac{f^{1 + \epsilon}}{l^{3 \epsilon}} \lvert  D \rvert^{1/2 - \epsilon}\\
		&\geq \frac{\epsilon^2}{30 \pi (\sqrt{2})^{1/ \epsilon}}  f^{1 - 5 \epsilon} \lvert  D \rvert^{1/2 - \epsilon}\\
		& \geq \frac{\epsilon^2}{30 \pi (\sqrt{2})^{1/ \epsilon}} \lvert f^2 D \rvert^{1/2 -  5 \epsilon / 2},
	\end{align*}
	where we use that $f^2 \geq l$ in the third inequality.
\end{proof}

	\section{Handling non-{$\epsilon$}-exceptional special points}
	
	The following result was proved by K\"uhne \cite[Corollary~2]{Kuhne13} in the case that $\epsilon \in (0, 1/10)$. 
	In general, it is a special case of a result proved using effective point counting by Binyamini \cite[Corollary~1]{Binyamini19}. 
	Strictly speaking, Binyamini \cite[\S2.2]{Binyamini19} fixes $\epsilon = 1/100$, but it is clear that his approach works for every $\epsilon \in (0, 1/2)$. 
	A simplified proof of \cite[Corollary~1]{Binyamini19} that avoids equidistribution may be given using the results of \cite{Binyamini24}.
	
	\begin{thm}[{\cite[Corollary~1]{Binyamini19}}]\label{thm:KB}
	Let $\epsilon \in (0, 1/2)$. 
			Let $V \subset \AB$ be a geometrically irreducible, non-special algebraic curve defined over a number field $L$. Write $X, Y$ for the coordinate functions on $\AB$ restricted to $V$.
		Let $d_1 = \deg X$ and $d_2 = \deg Y$.
		There exists an effectively computable constant $c(\epsilon, [L : \Q], \max\{d_1, d_2\} )$ with the following property: 
		if $(x, y) \in V$ is a special point with fundamental discriminants $(D_x, D_y) \neq (D_\epsilon, D_\epsilon)$, then
		\begin{align*}
			\max \{ \lvert \Delta_x \rvert, \lvert \Delta_y \rvert\} \leq c(\epsilon, [L : \Q], \max\{d_1, d_2\}).
		\end{align*} 
	\end{thm}

To prove Theorem~\ref{thm:main}, we therefore only need to consider special points $(x, y)$ with fundamental discriminants $D_x = D_y = D_\epsilon$.
 Such special points $(x, y)$ are either $\epsilon$-exceptional or not, and we treat the two cases separately.

In this section, we will strengthen Theorem~\ref{thm:KB} by proving the following theorem, which covers all special points $(x, y)$ that are not $\epsilon$-exceptional.
Note that this theorem holds for all non-special curves, and does not require the additional hypothesis of Theorem~\ref{thm:main} that $d_1 \neq d_2$. 

\begin{thm}\label{thm:nonB}
		Let $\epsilon \in (0, 1/2)$. 
	Let $V \subset \AB$ be a geometrically irreducible, non-special algebraic curve defined over a number field $L$. Write $X, Y$ for the coordinate functions on $\AB$ restricted to $V$.
	Let $d_1 = \deg X$ and $d_2 = \deg Y$.
	There exists an effectively computable constant $c(\epsilon, [L : \Q], \max\{d_1, d_2\} )$ with the following property: 
	if a special point $(x, y) \in V$ is not $\epsilon$-exceptional, then
	\begin{align*}
		\max \{ \lvert \Delta_x \rvert, \lvert \Delta_y \rvert\} \leq c(\epsilon, [L : \Q], \max\{d_1, d_2\}).
	\end{align*} 
\end{thm}

We will give two proofs of Theorem~\ref{thm:nonB}. 
The first proof uses effective point counting results of Binyamini \cite{Binyamini19, Binyamini24} and works for every $\epsilon \in (0, 1/2)$. 
In particular, this full range of $\epsilon$ is necessary to obtain the bound \eqref{eq:discbd} on the discriminants $\Delta_x, \Delta_y$ in Theorem~\ref{thm:main}.

The second proof adapts the method of K\"uhne \cite[Theorem~3]{Kuhne13} and is valid only if $\epsilon \in (0, 1/24)$. 
This is not sufficient to obtain the exponent $2 + \delta$ in the bound \eqref{eq:discbd} of Theorem~\ref{thm:main}.
It is however sufficient to obtain the bound \eqref{eq:htbd} on the heights $h(x), h(y)$ in Theorem~\ref{thm:main}. 
Moreover, K\"uhne's method is more conducive than the point counting approach to calculating the constant $c(\epsilon, [L : \Q], \max\{d_1, d_2\})$ in Theorem~\ref{thm:nonB} explicitly.

\subsection{Effective Pila--Zannier}

In \cite{Binyamini24}, Binyamini formulated a notion of an \textit{effectively o-minimal structure} \cite[\S1.5]{Binyamini24} and proved that the structure $\R_{\mathrm{LN}, \exp}$ is effectively o-minimal \cite[Theorem~3]{Binyamini24}.
The structure $\R_{\mathrm{LN}, \exp}$ is a reduct, in the sense of model theory, of the o-minimal structure $\R_{\mathrm{an}, \exp}$ used in many Diophantine applications of o-minimality.
The important fact \cite[\S10.4]{Binyamini24} for our purposes is that the restriction of the modular $j$-function to the set 
\[\mathfrak{F}_j = \{ z \in \h: \lvert z \rvert \geq 1 \mbox{ and } -1/2 \leq \re z \leq 1/2\}\] 
is definable in $\R_{\mathrm{LN}, \exp}$, under the natural identification of $\C$ with $\R^2$.

The Galois bound in Proposition~\ref{prop:nonexcept} suffices to prove that the non-$\epsilon$-exceptional special points $(x, y)$ are uniformly and effectively bounded via an effective version of the Pila--Zannier point counting strategy.

\begin{proof}[First proof of Theorem~\ref{thm:nonB}]
	By Theorem~\ref{thm:KB}, we only need to consider  non-$\epsilon$-exceptional special points $(x, y) \in V$ for which $D_x = D_y$.
	We apply the Pila--Zannier point counting strategy in the effectively o-minimal structure $\R_{\mathrm{LN}, \exp}$. 
	Proposition~\ref{prop:nonexcept} provides the necessary effective lower bound for the Galois orbit of non-$\epsilon$-exceptional special points $(x, y)$ with $D_x = D_y$. 
\end{proof}

\subsection{K{\"u}hne's method}

 Alternatively, if $\epsilon$ is chosen suitably small, one can adapt K\"uhne's proof of \cite[Theorem~3]{Kuhne13}. 
In this way, Theorem~\ref{thm:nonB} for $\epsilon \in (0, 1/24)$ follows from Theorem~\ref{thm:KB} and the following proposition. 

\begin{prop}\label{prop:nonK}
	Let $\epsilon \in (0, 1/24)$. 
	Let $V \subset \AB$ be a geometrically irreducible, non-special algebraic curve defined over a number field $L$. Write $X, Y$ for the coordinate functions on $\AB$ restricted to $V$.
	Let $d_1 = \deg X$ and $d_2 = \deg Y$.
	There exists an effectively computable constant $c(\epsilon, [L : \Q], \max\{d_1, d_2\} )$ with the following property.
Let $x, y$ be singular moduli of the same fundamental discriminant $D$ such that the special point $(x, y)$ is not $\epsilon$-exceptional.
 If $(x, y) \in V$, then 
	\begin{align*}
		\max \{ \lvert \Delta_x \rvert, \lvert \Delta_y \rvert\} \leq c(\epsilon, [L : \Q], \max\{d_1, d_2\}).
	\end{align*}
\end{prop}

\begin{proof}
	All constants $c_i(\ldots)$ will be effectively computable.
	If $0 \in \{d_1, d_2\}$, then one of the coordinates of $V$ is constant. 
	Since $V$ is not a special curve, this constant coordinate is not a singular modulus. 
	Therefore, $V$ does not contain any special points in this case.
	So assume subsequently that $d_1, d_2 > 0$.
	
	Without loss of generality, we may assume that $\lvert \Delta_x \rvert \geq \lvert \Delta_y \rvert$.
	Since $(x, y)$ is not $\epsilon$-exceptional, Proposition~\ref{prop:degnonexcept} implies that
	\[ [\Q(x) : \Q] \geq c_1(\epsilon) \left\lvert \Delta_x \right\rvert^{1/2 -  5\epsilon/2}.\]
	Certainly then
	\begin{align*}
	[L(x, y) : \Q] \geq 	c_1(\epsilon) \left\lvert \Delta_x \right\rvert^{1/2 -  5\epsilon/2}.
		\end{align*}
	Since $(x, y) \in V$ and $V$ is defined over $L$ with $d_1, d_2 > 0$, we obtain that
	\[ [\Q(y) : \Q] \geq c_2(\epsilon, [L : \Q], \max \{d_1, d_2\}) \left\lvert \Delta_x \right\rvert^{1/2 - 5\epsilon/2}.\]
	
	We may now emulate the proof of \cite[Theorems~2 \& 3]{Kuhne13}. In particular, \cite[Lemma~3]{Kuhne13} (applied with the same $\epsilon$) implies that
	\[ \max \{h(x), h(y) \} \leq c_3(\epsilon, [L : \Q], \max \{d_1, d_2\}) \lvert \Delta_x \rvert^{3 \epsilon}.\]
	And \cite[Theorem~1.1]{Wustholz14} gives that 
	\[ \lvert \Delta_x \rvert \leq c_4([L : \Q], \max \{d_1, d_2\}) \max \{1, h(V)\}^{8}.\]
	Therefore, combining these two inequalities, we obtain that
	\begin{align}\label{eq:smallht}
		 \max \{h(x), h(y) \} \leq c_5(\epsilon, [L : \Q], \max \{d_1, d_2\}) \max \{1, h(V)\}^{24 \epsilon}.
		 \end{align}
	
	Since $\epsilon < 1/24$, the bound \eqref{eq:smallht} implies that the point $(x, y)$ is a ``point of small height'' on $V$, i.e.~$(x, y)$ satisfies the inequality in \cite[(3.1)]{Kuhne13} if $h(V)$ is suitably large (which we may always assume). 
	The number of ``small height'' points on $V$ is bounded by a constant $c_6(\epsilon, [L : \Q], \max \{d_1, d_2\})$, thanks to \cite[Lemma~2]{Kuhne13}. 
	Proposition~\ref{prop:nonexcept} therefore implies that
	\[ \max \{ \lvert \Delta_x \rvert, \lvert \Delta_y \rvert\} \leq c_7(\epsilon, [L : \Q], \max \{d_1, d_2\}). \qedhere\]
\end{proof}

\section{Handling {$\epsilon$}-exceptional special points}

To complete the proof of Theorem~\ref{thm:main}, it remains to deal with special points $(x, y) \in V$ that are $\epsilon$-exceptional.

\subsection{From height bounds to finiteness}\label{sec:G}

Before proving Theorem~\ref{thm:main}, we will show how one may deduce a version of Andr\'e's theorem directly from Theorem~\ref{thm:nonB} and a height bound for special points of the form of \eqref{eq:PapHt}.

Let $V \subset \AB$ be an irreducible, non-special algebraic curve defined over $\alg$.
Suppose that constants $c_1, c_2 > 0$ are such that 
\begin{align}\label{eq:Gbd}
	\max \{ h(x), h(y)\} \leq c_1 [\Q(x, y) : \Q]^{c_2}
\end{align}
for every special point $(x, y) \in V$.
Such a bound follows, for example, from a result of Papas \cite[Theorem~2.17]{Papas26} in the case that $V$ satisfies a suitable boundary condition.
A conjecture of Habegger \cite[Conjecture]{Habegger10} would imply, for every non-special curve $V$, a stronger bound
\[\max \{ h(x), h(y)\} \leq c_3 \log (1 + [\Q(x, y) : \Q]). \]

Let $\epsilon \in (0, 1/2)$ be such that $\epsilon > c_2/(2 c_2 + 1)$. 
Let $(x, y) \in V$ be an $\epsilon$-exceptional special point.
Then Propositions~\ref{prop:except} and \ref{prop:htexcept}, together with \eqref{eq:Gbd}, imply that
\[ \frac{30 \pi}{\epsilon} \max \{\lvert \Delta_x \rvert, \lvert \Delta_y \rvert\}^\epsilon \leq c_1 \left( \frac{\epsilon}{10 \pi} \max \{\lvert \Delta_x \rvert, \lvert \Delta_y \rvert\}^{(1/2 -  \epsilon)} \right)^{ 2 c_2}. \]
Note that $\epsilon - (1 - 2 \epsilon) c_2 > 0$.
Hence, 
\[ \max \{ \lvert \Delta_x \rvert, \lvert \Delta_y \rvert\} \leq \left(\frac{c_1}{3} \left(\frac{\epsilon}{10 \pi}\right)^{2 c_2 + 1} \right)^{1/(\epsilon - (1 - 2 \epsilon) c_2)}.\]
On the other hand, if a special point $(x, y) \in V$ is not $\epsilon$-exceptional, then $\max \{ \lvert \Delta_x \rvert, \lvert \Delta_y \rvert\}$ is bounded by Theorem~\ref{thm:nonB}.

\begin{remark}
	Of course, since Andr\'e \cite{Andre98} showed that a non-special curve contains only finitely many special points, there will always exist, for a given curve $V$, constants $c_1, c_2$ such that \eqref{eq:Gbd} holds.
	The case where the above argument is interesting is therefore when $c_1, c_2$ either are of relatively modest size or are uniform in certain data associated to $V$.
	Then we may deduce, by the above argument, a very strong bound on the discriminants of $\epsilon$-exceptional special points that lie on $V$, without appealing to Masser--W\"ustholz endomorphism estimates or point counting.
	Together with Theorem~\ref{thm:nonB}, this establishes an explicit version of Andr\'e's theorem for such curves $V$. 
	In particular, it gives a different way to deduce \cite[Corollary~1.7]{Papas26} from \cite[Theorem~2.17]{Papas26} than the one used by Papas.
\end{remark}

\subsection{Bounds for $\epsilon$-exceptional special points}

We now start the proof of Theorem~\ref{thm:main}.
The following lemma is a convenience to ease notation later.

\begin{lem}\label{lem:FaltHtExcept}
	Let $\epsilon \in (0, 1/2)$. 
	Let $(x, y)$ be an $\epsilon$-exceptional special point. 
	Assume that $\lvert \Delta_x \rvert \geq \lvert \Delta_y \rvert$. 
	Let $E_x$ be an elliptic curve with $j$-invariant $x$. 
	If
	\[ \falt(E_x) \leq \log [\Q(x, y) : \Q], \]
	then 
	\[ \lvert \Delta_x \rvert \leq \left( \frac{1}{12 \epsilon} \right)^{2/\epsilon}.\]
\end{lem}

\begin{proof}
	Since $(x, y)$ is $\epsilon$-exceptional, Proposition~\ref{prop:except} implies that
	\begin{align*} 
		[\Q(x, y) : \Q] \leq \cl(\Delta_x) \cl(\Delta_y)	\leq \left( \frac{\epsilon}{10 \pi} \lvert \Delta_x \rvert^{1/2 - \epsilon}\right)^2.
		\end{align*}
Pazuki's explicit version \cite[(3.19)]{Pazuki19} of Silverman's \cite[Proposition~2.1]{Silverman86} height comparison gives that
	\[ \falt(E_x) \geq \frac{1}{12} h(x) - \frac{1}{2} \log (1 + h(x)) - 2.08.\]
	Note that
	\[ \frac{1}{12} t - \frac{1}{2} \log ( 1 + t) - 2.08 \geq \frac{1}{20} t - 3\]
	for every $t \geq 0$.
	 Proposition~\ref{prop:htexcept} therefore implies that
	\[ \falt(E_x) \geq \frac{3 \pi}{2 \epsilon} \lvert \Delta_x \rvert^\epsilon -3.\]

	Suppose then that
	\[ \falt(E_x) \leq \log [\Q(x, y) : \Q]. \]
	Then
	\[ \frac{3 \pi}{2 \epsilon} \lvert \Delta_x \rvert^\epsilon -3 \leq 2 \log \left ( \frac{\epsilon}{10 \pi} \right) + (1 - 2 \epsilon) \log ( \lvert \Delta_x \rvert).\]
	Since $\epsilon \in (0, 1/2)$, we obtain that
	\[ 3 \pi \lvert \Delta_x \rvert^\epsilon - \log (\lvert \Delta_x \rvert) \leq 3 - 2 \log \left( 20 \pi \right) \leq 0.\]
	Recall that if $\delta>0$, then $\log(t) \leq t^\delta/(e \delta)$ for every $t > 0$. Therefore,
	\[ 3 \pi \lvert \Delta_x \rvert^{\epsilon } \leq  \frac{2}{e \epsilon} \lvert \Delta_x \rvert^{\epsilon/2 }. \qedhere\]
\end{proof}

The proof of Theorem~\ref{thm:main} rests on the following theorem.
The hypothesis that $d_1 \neq d_2$ is required in the proof only for the appeal to Proposition~\ref{prop:wkbd}.

\begin{thm}\label{thm:except}
	Let $\epsilon \in (0, 1/2)$. 
	Let $V \subset \AB$ be a geometrically irreducible, non-special algebraic curve defined over a number field $L$. Write $X, Y$ for the coordinate functions on $\AB$ restricted to $V$.
	Let $d_1 = \deg X$ and $d_2 = \deg Y$.
	Suppose that $d_1, d_2 > 0$ and $d_1 \neq d_2$.
	Let $(x, y) \in V$ be an $\epsilon$-exceptional special point. 
	If
		\[\max \{ h(x), h(y)\} \geq 26 \max \{d_1, d_2\}^3 (L_{d_1, d_2} + h(V))\]
		where $L_{d_1, d_2} = \log (2^{\min \{d_1, d_2\}} (d_1 + 1) (d_2 + 1))$, then
			\[ \max \{\lvert \Delta_x \rvert , \lvert \Delta_y \} \rvert \leq \max \left\{\frac{1}{12 \epsilon}, 37 + 3 \log \max \{d_1, d_2\}, 6 \max \{d_1, d_2\}\right\}^{2/\epsilon}.\]
\end{thm}

\begin{proof}
Denote by $D_x, D_y$ the fundamental discriminants of $x, y$ respectively. 
Note that $D_x = D_y$, since $(x, y)$ is $\epsilon$-exceptional. Hence, Proposition~\ref{prop:modpoly} implies that $\Phi_N(x, y) = 0$ for some $N \in \Z_{>0}$, which we may take to be the minimal such $N$ with this property. Without loss of generality, assume that $\lvert \Delta_x \rvert \geq \lvert \Delta_y \rvert$.
	
	Suppose that
	\begin{align}\label{eq:assume}
		\max \{ h(x), h(y)\} \geq 26 \max \{d_1, d_2\}^3 (L_{d_1, d_2} + h(V)).
		\end{align}	
	Since $(x, y) \in V$ and $\Phi_N(x, y) = 0$, Proposition~\ref{prop:wkbd} therefore implies that 
	\begin{align}\label{eq:htup} 
		h(x), h(y) \leq  \max \{d_1, d_2\} (474 + 618  \log N).
		\end{align}
	Since $(x, y)$ is $\epsilon$-exceptional, Proposition~\ref{prop:htexcept} implies that
	\begin{align}\label{eq:htlow}
		 h(x) \geq \frac{30 \pi}{\epsilon} \lvert \Delta_x \rvert^\epsilon \mbox{ and } h(y) \geq \frac{30 \pi}{\epsilon} \lvert \Delta_y \rvert^\epsilon.
		 \end{align}
	From \eqref{eq:htup} and \eqref{eq:htlow}, we obtain that
	\begin{align}\label{eq:log}
		\log N &\geq  \left( \frac{0.15}{\epsilon \max \{d_1, d_2\}} \lvert \Delta_x \rvert^\epsilon- 0.8 \right).
	\end{align}

	Let $E_x / \alg$ be an elliptic curve with $j$-invariant $x$.
	We may assume that
	\[ \falt(E_x) \geq \log [\Q(x, y) : \Q],\]
	since otherwise Lemma~\ref{lem:FaltHtExcept} would imply that
	\begin{align}\label{eq:falt}
		 \lvert \Delta_x \rvert \leq \left( \frac{1}{12 \epsilon}\right)^{2/\epsilon}.
		 \end{align}
	The bound $\falt(E_x) +1.18 \leq h(x)/12 $ in \cite[(3.19)]{Pazuki19} then implies that
	\[ \max \left\{1, \falt(E_x),  \log [\Q(x, y) : \Q]\right\} \leq \frac{1}{12} h(x).\]
	Hence, from Proposition~\ref{prop:isog} and \eqref{eq:htup} we obtain that
	\begin{align*} N &\leq 10^{13} \left(\frac{1}{12} h(x) \right)^2 \left([\Q(x, y) : \Q] \right)^2\\
		&\leq 10^{13} \left(\frac{1}{12} \max \{d_1, d_2\} (474 + 618 \log N) \right)^2 \left([\Q(x, y) : \Q] \right)^2.
		\end{align*}
	Therefore,
	\begin{align}\label{eq:deglow}
		[\Q(x, y) : \Q] &\geq \frac{12 }{ 10^{13/2} \max \{d_1, d_2\} } \frac{N^{1/2}}{474 + 618 \log N} \nonumber\\
		&\geq \frac{1.8 N^{1/3} }{ 10^{9} \max \{d_1, d_2\} } ,
		\end{align}
	since, for every $t \geq 1$, it is elementary that
	\[ \frac{t^{1/2}}{474 + 618 \log t} \geq \frac{t^{1/3}}{2000}.\]

	 On the other hand, since $(x, y)$ is $\epsilon$-exceptional,
	\begin{align}\label{eq:degup}
		[\Q(x, y) : \Q] &\leq \cl(\Delta_x) \cl(\Delta_y) \leq \frac{\epsilon^2}{100 \pi^2} \lvert \Delta_x \rvert^{1 - 2 \epsilon},
		\end{align}
	 by Proposition~\ref{prop:except} and the fact that $\lvert \Delta_x \rvert \geq \lvert \Delta_y \rvert$.
From inequalities \eqref{eq:deglow} and \eqref{eq:degup} we therefore obtain that
\[\frac{1.8 N^{1/3} }{ 10^{9} \max \{d_1, d_2\} }  \leq \frac{\epsilon^2}{100 \pi^2} \lvert \Delta_x \rvert^{1 - 2 \epsilon}.\]
Consequently, 
	\[ \log N \leq 3 \left( \log \left( \frac{10^9 \max \{d_1, d_2\} \epsilon^2}{180 \pi^2} \right) + \left(1 - 2 \epsilon \right) \log \lvert \Delta_x \rvert \right).\]
	Since $\epsilon \in (0, 1/2)$, we may deduce that
	\begin{align}\label{eq:log2}
		 \log N \leq 3 \left( 12 + \log  \max \{d_1, d_2\} + \log \lvert \Delta_x \rvert \right).
		 \end{align}
	
	Then the inequalities \eqref{eq:log} and \eqref{eq:log2} imply that
	\[\left( \frac{0.15}{\epsilon \max \{d_1, d_2\}} \lvert \Delta_x \rvert^\epsilon- 0.8 \right) \leq 3 \left( 12 + \log \max \{d_1, d_2\} + \log \lvert \Delta_x \rvert \right). \]
	And so
	\[ \frac{0.15}{\epsilon \max \{d_1, d_2\}} \lvert \Delta_x \rvert^\epsilon - 3 \log \lvert \Delta_x \rvert \leq 37 + 3 \log  \max\{d_1, d_2\} .\]
	Since $\log(t) \leq t^{\epsilon}/e \epsilon$ for every $t>0$, we obtain that
	\[ \frac{0.15}{ \epsilon \max \{d_1, d_2\}} \lvert \Delta_x \rvert^{\epsilon} - \frac{6}{ e \epsilon} \lvert \Delta_x \rvert^{\epsilon/2} \leq 37 + 3 \log \max\{d_1, d_2\} .\]
	And so
	\[ \lvert \Delta_x \rvert^{\epsilon/2} \left( \frac{0.15}{ \epsilon \max \{d_1, d_2\}} \lvert \Delta_x \rvert^{\epsilon/2} - \frac{6}{ e \epsilon} \right) \leq 37 + 3 \log \max\{d_1, d_2\}. \]
	Thus, either
	\begin{align}\label{eq:bd1} 
		\lvert \Delta_x \rvert \leq \left( \frac{ \epsilon \max \{d_1, d_2\}}{0.15} \left( 1 + \frac{6}{ e \epsilon}\right) \right)^{2/\epsilon} \leq (6 \max \{d_1, d_2\})^{2/\epsilon},
		\end{align}
	or
	\begin{align}\label{eq:bd2}
		\lvert \Delta_x \rvert \leq (37 + 3 \log \max \{d_1, d_2\})^{2/\epsilon}.
		\end{align}
	Taking the maximum of \eqref{eq:falt}, \eqref{eq:bd1}, and \eqref{eq:bd2}, we obtain that, in any case,
		\begin{align*}
			\lvert \Delta_x \rvert \leq \max \left\{\frac{1}{12 \epsilon}, 37 + 3 \log \max \{d_1, d_2\}, 6 \max \{d_1, d_2\}\right\}^{2/\epsilon}
			\end{align*}
	if \eqref{eq:assume} holds.
\end{proof}

\begin{remark}\label{rmk:sharp}
	The constants in Theorem~\ref{thm:except} are completely explicit.
	Moreover, they are independent of $[L : \Q]$.
If $\epsilon$ is not too small, then the constants are also of relatively modest size in terms of $\max\{d_1, d_2\}$ and $h(V)$.
Theorem~\ref{thm:except} therefore offers a practical way to rule out the existence of $\epsilon$-exceptional special points on many curves of interest.
	
	For example, fix $\epsilon = 2/5$. 
	Let $V$ be an irreducible algebraic curve satisfying the hypotheses of Theorem~\ref{thm:except} and such that $\max \{d_1, d_2\} \leq 10$ and $h(V) \leq 10$.
	Then Theorem~\ref{thm:except} and Proposition~\ref{prop:htexcept} together imply that any 
	$\epsilon$-exceptional special point $(x, y) \in V$ satisfies
	\[\max \{ \lvert \Delta_x \rvert, \lvert \Delta_y \rvert\} < 8 \times 10^8.\]
	
	It is easy to verify computationally (e.g.~in PARI \cite{PARI24}) that every fundamental discriminant $D$ with $\lvert D \rvert < 8 \times 10^8$ satisfies the bound \eqref{eq:Tat} with $\epsilon = 2/5$, i.e.~none of the fundamental discriminants $D$ with $\lvert D \rvert < 8 \times 10^8$ is the single possible exceptional fundamental discriminant in Tatuzawa's theorem applied with $\epsilon = 2/5$.
	In particular, there are therefore no $\epsilon$-exceptional special points $(x, y)$ with $\max \{ \lvert \Delta_x \rvert, \lvert \Delta_y \rvert\} < 8 \times 10^8$.
	Hence, there are no $\epsilon$-exceptional special points $(x, y) \in V$.
	(Of course, this may not surprise the reader who believes in the Generalised Riemann Hypothesis.)
\end{remark}

\subsection{The proof of Theorem~\ref{thm:main}}

We are now ready to prove Theorem~\ref{thm:main}.

\begin{proof}[Proof of Theorem~\ref{thm:main}]
	Let $V \subset \AB$ be a geometrically irreducible, non-special algebraic curve defined over a number field $L$. Write $X, Y$ for the coordinate functions on $\AB$ restricted to $V$.
	Let $d_1 = \deg X$ and $d_2 = \deg Y$.
	As in the proof of Proposition~\ref{prop:nonK}, we may assume that $d_1, d_2 > 0$.
	Suppose that $d_1 \neq d_2$.
	
	Let $\epsilon \in (0, 1/2)$. 
	By Theorems~\ref{thm:KB} and \ref{thm:nonB}, there exists an effectively computable constant $c_1([L : \Q], \max \{d_1, d_2\}, \epsilon)$ such that
	\[ \max \{ \lvert \Delta_x \rvert, \lvert \Delta_y \rvert\} \leq c_1([L: \Q], \max \{d_1, d_2\}, \epsilon)\]
	for every special point $(x, y) \in V$ that is not $\epsilon$-exceptional.
	Hence, for every special point $(x, y) \in V$ that is not $\epsilon$-exceptional,
	\[ \max \{ h(x), h(y) \} \leq \frac{9 c_1([L: \Q], \max \{d_1, d_2\}, \epsilon)^{1/2}}{2} \]
	by \cite[Lemma~2.10]{Riffaut19}.
	
	If a special point $(x, y) \in V$ is $\epsilon$-exceptional, then Theorem~\ref{thm:except} implies that either
	\[ \max \{h(x), h(y)\} \leq 26 \max \{d_1, d_2\}^3 (L_{d_1, d_2} + h(V))\]
	where $L_{d_1, d_2} = \log (2^{\min \{d_1, d_2\}} (d_1 + 1) (d_2 + 1))$, or
		\[ \max \{ \lvert \Delta_x \rvert, \lvert \Delta_y \rvert\} \leq \max \left\{\frac{1}{12 \epsilon}, 37 + 3 \log \max \{d_1, d_2\}, 6 \max \{d_1, d_2\}\right\}^{2/\epsilon}.\]
		In the latter case,
		\[ \max \{ h(x), h(y) \} \leq \frac{9}{2} \max \left\{\frac{1}{12 \epsilon}, 37 + 3 \log \max \{d_1, d_2\}, 6 \max \{d_1, d_2\}\right\}^{1 + \epsilon/2} \]
	by \cite[Lemma~2.10]{Riffaut19} again.	
	
So there is an effectively computable constant $c_2([L : \Q], \max \{d_1, d_2\}, \epsilon)$ such that
	\[\max \{ h(x), h(y) \} \leq c_2([L : \Q], \max \{d_1, d_2\}, \epsilon) + 26 \max \{d_1, d_2\}^3 h(V) \]
	for every special point $(x, y) \in V$.
	The bound \eqref{eq:htbd} in Theorem~\ref{thm:main} is then just the specialisation to some arbitrary $\epsilon \in (0, 1/2)$.
	
	Now we prove the bound \eqref{eq:discbd}. Let some $\delta > 0$ be given.
	Fix
	\[ \epsilon = \frac{1}{2 + \delta/2},\]
	so that $\epsilon \in (1/(2 + \delta), 1/2)$.
	By the above proof, there exist effectively computable constants $c_3([L : \Q], \max \{d_1, d_2\}, \delta)$ and $c_4([L : \Q], \max \{d_1, d_2\}, \delta)$ with the following property:
for every special point $(x, y) \in V$,
	\[ \max \{ \lvert \Delta_x \rvert, \lvert \Delta_y \rvert\} \leq c_3([L : \Q], \max \{d_1, d_2\}, \delta)\]
	if $(x, y)$ is not $\epsilon$-exceptional, and
	\[ \max \{ h(x), h(y)\} \leq c_4([L : \Q], \max \{d_1, d_2\}, \delta) \max \{1, h(V)\}\]
	if $(x, y)$ is $\epsilon$-exceptional.

	If a special point $(x, y)$ is $\epsilon$-exceptional, then
	\[ h(x) \geq \frac{30 \pi}{\epsilon} \lvert \Delta_x \rvert^\epsilon \mbox{ and } h(y) \geq \frac{30 \pi}{\epsilon} \lvert \Delta_y \rvert^\epsilon\]
	by Proposition~\ref{prop:htexcept}.
	Therefore, since $\epsilon > 1/(2+ \delta)$, there exists an effectively computable constant $c_5([L : \Q], \max \{d_1, d_2\}, \delta)$ such that 
	\[ \max \{ \lvert \Delta_x \rvert, \lvert \Delta_y \rvert\} \leq c_5([L : \Q], \max \{d_1, d_2\}, \delta) \max \{1, h(V)\}^{2 + \delta}\]
	for every $\epsilon$-exceptional special point $(x, y) \in V$.
	This proves the bound \eqref{eq:discbd}.
	Note that, to prove this for every $\delta > 0$, it was crucial that Theorem~\ref{thm:nonB} holds for every $\epsilon \in (0, 1/2)$, and not just for $\epsilon \in (0, 1/24)$ as in Proposition~\ref{prop:nonK}.
\end{proof}

\section{Proof of Theorem~\ref{thm:WBH}}

Finally, we come to the proof of Theorem~\ref{thm:WBH}.
The first part follows the same approach as Theorem~\ref{thm:main}.

\begin{proof}[Proof of Theorem~\ref{thm:WBH}]
	
	Suppose that $\mathrm{WBH}(\eta, \kappa)$ holds for some $\eta \in (0, 1/2)$ and $\kappa \in \R_{\geq 0}$ with a given constant $c_*$.
	Let $\delta > 0$.
	Fix some $\epsilon \in (0, 1/2)$ for which
	\begin{align}\label{eq:epsilon}
	0 <	\frac{\kappa}{4 \eta}  \frac{1}{\epsilon - (1 - 2 \epsilon)(1/2 \eta -1)} \leq \frac{\kappa}{2 \eta} + \delta,
		\end{align}
	which we may always achieve by taking $\epsilon$ sufficiently close to $1/2$.
	To limit the proliferation of constants, we will use Vinogradov's asymptotic notation $\ll$, where the implied constants will always be effectively computable in terms of $[L : \Q]$, $\max \{d_1, d_2\}$, $c_*$, and $\delta$.
	
	Thanks to Theorem~\ref{thm:nonB}, it is enough to consider only special points $(x, y) \in V$ that are $\epsilon$-exceptional.
	Let $(x, y) \in V$ be an $\epsilon$-exceptional special point.
	Assume, without loss of generality, that $\lvert \Delta_x \rvert \geq \lvert \Delta_y \rvert$.
	Let 
	\[ d = [\Q(x, y) : \Q].\]
	Let $E_x$ be an elliptic curve with $j$-invariant $x$.
	By Lemma~\ref{lem:FaltHtExcept}, we may assume that $\falt(E_x) \geq \log d$. 
	Hence, by \cite[(3.19)]{Pazuki19},
	\begin{align}\label{eq:ht1}
		 h(x) \gg \log d.
		 \end{align}
	
	Since $(x, y)$ is $\epsilon$-exceptional, we have that $D_x = D_y$.
	 Proposition~\ref{prop:modpoly} then implies that $\Phi_N(x, y) = 0$ for some $N \in \Z_{> 0}$, which we take to be the least such $N$.
	 Since we are assuming that $\mathrm{WBH}(\eta, \kappa)$ holds, we have that
	 \begin{align}\label{eq:ht2}
	 	 \max \{h(x), h(y)\} \ll \max \{1, h(V)\}^\kappa N^{1/2 - \eta}.
	 	 \end{align}
 	 
 	 Proposition~\ref{prop:isog} together with \eqref{eq:ht1} and \eqref{eq:ht2} imply that
 	 \[ N \ll \left(\max\{1, h(V)\}^\kappa N^{1/2 - \eta} d\right)^2.\]
 	 So
 	 \begin{align}\label{eq:N1}
 	 	 N^\eta \ll \max \{1, h(V)\}^\kappa d.
 	 	 \end{align}
 	 
 	 Since $(x, y)$ is $\epsilon$-exceptional, Proposition~\ref{prop:htexcept} implies that
 	 \begin{align}\label{eq:ht3}
 	 	 h(x) \gg \lvert \Delta_x \rvert^\epsilon,
 	 	 \end{align}
 	 and Proposition~\ref{prop:except} implies, since $\lvert \Delta_y \rvert \leq \lvert \Delta_x \rvert$, that
 	 \begin{align}\label{eq:deg1}
 	 	 d \ll \lvert \Delta_x \rvert^{1 - 2 \epsilon}.
 	 	 \end{align}
  	 
  	 Now put these inequalities together. 
  	 From \eqref{eq:ht2} and \eqref{eq:ht3}, we obtain that
  	 \[ \lvert \Delta_x \rvert^\epsilon \ll \max \{1, h(V)\}^\kappa N^{1/2 - \eta}.\]
  	 So, by \eqref{eq:N1},
  	 \[ \lvert \Delta_x \rvert^\epsilon \ll \max \{1, h(V)\}^\kappa \left(\left(\max \{1, h(V)\}^\kappa d\right)^{1/\eta}\right)^{1/2 - \eta} .\]
  	 Hence, by \eqref{eq:deg1},
  	 \begin{align*} 
  	 	\lvert \Delta_x \rvert^\epsilon &\ll \max \{1, h(V)\}^\kappa \left(\max \{1, h(V)\}^\kappa \lvert \Delta_x \rvert^{1 - 2 \epsilon}\right)^{1/2\eta -1} \\
  	 	&\ll \max \{1, h(V)\}^{\kappa/2 \eta} \lvert \Delta_x \rvert^{(1 - 2 \epsilon)(1/2\eta -1)}.
  	 	\end{align*}
   	So
   	\[ \lvert \Delta_x \rvert^{\epsilon - (1 - 2 \epsilon)(1/2 \eta - 1)} \ll \max \{1, h(V)\}^{\kappa/2 \eta}.\]
   	By \cite[Lemma~2.9]{Riffaut19},
   	\[ \max\{ h(x), h(y)\} \ll \lvert \Delta_x \rvert^{1/2}.\]
   	Therefore,
   	\[\max\{ h(x), h(y)\} \ll (\max \{1, h(V)\}^{\kappa/4 \eta})^{1/(\epsilon - (1 - 2 \epsilon)(1/2 \eta -1))}.\]
	This completes the proof of (1) in Theorem~\ref{thm:WBH}, thanks to inequality \eqref{eq:epsilon}.
	
	Now we prove (2) of Theorem~\ref{thm:WBH}. 
	Suppose additionally that $\kappa < 2 \eta$. Let
	\[ \delta = \frac{1- \frac{\kappa}{2 \eta}}{2}.\]
	By (1) of Theorem~\ref{thm:WBH}, there exists an effectively computable constant $c_0 = c_0([L : \Q], \max \{d_1, d_2\}, c_*)$ such that, for every special point $(x, y) \in V$,
	\begin{align}\label{eq:Kht}
		  \max \{h(x), h(y)\} \leq c_0 \max \{1, h(V)\}^{\kappa/2 \eta + \delta}.
		  \end{align}

	Note that $\kappa/2\eta + \delta < 1$.
	We may therefore follow the argument of K\"uhne \cite[\S5]{Kuhne13}.
	By \cite[Lemma~2]{Kuhne13} and \cite[(2.1)]{Kuhne13}, there exist explicit constants $c_1 = c_1(\max \{d_1, d_2\}) > 0$ and $c_2 = c_2(\max \{d_1, d_2\})$ such that the number of points $(x, y) \in V(\alg)$ for which
	\begin{align}\label{eq:Uht}
		 \max \{h(x), h(y)\} \leq c_1 h(V) + c_2 
		 \end{align}
	is bounded by an explicit constant $c_3 = c_3(\max \{d_1, d_2\})$.
	
	If $h(V)$ is larger than some constant $c_4([L : \Q], \max \{d_1, d_2\}, c_*)$, then \eqref{eq:Kht} implies \eqref{eq:Uht}.
	In this case, there are therefore at most $c_3$ many special points $(x, y) \in V$. 
	Hence, for every special point $(x, y) \in V$,
	\[\max \{ \lvert \Delta_x \rvert, \lvert \Delta_y \rvert\} \leq c_5([L : \Q], \max \{d_1, d_2\})\] 
	thanks to Proposition~\ref{prop:Gold}.
	If $h(V) \leq c_4([L : \Q], \max \{d_1, d_2\}, c_*)$, then \cite[Theorem~2]{Kuhne12} implies directly that
	\[ \max \{ \lvert \Delta_x \rvert, \lvert \Delta_y \rvert\} \leq c_6([L : \Q], \max \{d_1, d_2\}, c_*)\]
	for every special point $(x, y) \in V$.
	In either case, we are done.
\end{proof}


\end{document}